\documentclass[10pt]{article}

\usepackage[T1]{fontenc}
\usepackage{amsmath}
\usepackage{amsthm}
\usepackage{amssymb}
\usepackage[square,comma,numbers,sort&compress]{natbib}

\usepackage{tikz}
\usetikzlibrary{positioning}
\tikzset{
    every path/.style={thick},
    every node/.style={minimum width=5pt, inner sep=1pt},
    hollow/.style={shape=circle, fill=white, node contents={}, outer sep=0pt, draw},
    letter/.style={text height=1.5ex, text depth=.5ex}
}
\newcommand{\arc}[2]{
    \draw (#1,0)
        .. controls ({#1},{0.5*(#2-#1)})
                    and ({#2},{0.5*(#2-#1)})
        .. (#2,0);
}

\usepackage[small, pagestyles]{titlesec}
\usepackage[small, it]{caption}
\usepackage{parskip}

\theoremstyle{plain}
\newtheorem{theorem}{Theorem}[section]
\newtheorem{proposition}[theorem]{Proposition}
\newtheorem{observation}[theorem]{Observation}

\renewenvironment{abstract}%
                {
                  \begin{list}{}%
                     {\setlength{\rightmargin}{1in}%
                      \setlength{\leftmargin}{1in}}%
                   \item[]\ignorespaces\begin{small}}%
                 {\end{small}\unskip\end{list}}

\usepackage[bookmarks]{hyperref}
\hypersetup{
	colorlinks=true,
	linkcolor=black,
	anchorcolor=black,
	citecolor=black,
	urlcolor=black,
	pdfpagemode=UseThumbs,
	pdftitle={Deciding lettericity is NP-complete, even for four-colorable comparability graphs},
	pdfsubject={Graph Theory},
	pdfauthor={Vatter},
}

\newpagestyle{main}[\small]{
        \headrule
        \sethead[\usepage][][]
        {\textsc{Deciding Lettericity Is NP-Complete}}{}{\usepage}}

\title{\textsc{Deciding Lettericity Is NP-Complete, Even for Four-Colorable Comparability Graphs}}
\author{%
	Vincent Vatter\\
	\small Department of Mathematics\\[-0.5ex]
	\small University of Florida\\[-0.5ex]
	\small Gainesville, Florida USA
}
\date{}

\begin{document}
\maketitle
\pagestyle{main}

\begin{abstract}
We prove that deciding whether a graph has lettericity at most $k$ is NP-complete, even for four-colorable comparability graphs. Our reduction maps a bipartite graph $G$ to the graph obtained from its incidence graph by inflating each vertex by a clique or an independent set of three vertices. The lettericity of this graph is determined by the numbers of vertices and edges of $G$ and the greatest number of edge-disjoint paths on four vertices in $G$. Teypaz and Rapine showed that deciding whether the edges of a bipartite graph can be partitioned into paths on four vertices is NP-complete.
\end{abstract}

\section{Introduction}
\label{sec:intro}

Lettericity was introduced by Petkov\v{s}ek~\cite{petkovsek:letter-graphs-a:} to study well-quasi-order. He showed that, for every~$k$, the graphs of lettericity at most $k$ are well-quasi-ordered under the induced subgraph order. Lettericity has been studied in several subsequent papers~\cite{ferguson:on-the-letteric:,alecu:critical-proper:,mandrick:bounds-on-the-l:,alecu:letter-graphs-a:abstract,alecu:letter-graphs-a:}, in part because of its connection to the geometric grid classes of permutations~\cite{albert:geometric-grid-:,albert:inflations-of-g:,albert:inflations-of-g:2x4:,bevan:growth-rates-of:geom,braunfeld:decidability-in:}. In some sense, hereditary properties of graphs with bounded lettericity are the graphical analogues of geometric grid classes~\cite{alecu:letter-graphs-a:iff}.

All graphs in this paper are finite and simple. For a finite alphabet $\Sigma$, a \emph{decoder} is a set $D \subseteq \Sigma^2$ of ordered pairs. For a word $w = w(1) w(2) \cdots w(n)$ over $\Sigma$, the \emph{letter graph} of~$w$ with respect to $D$ is the graph $\Gamma_D(w)$ with vertex set $\{1, 2, \dots, n\}$ in which, for $i < j$, the vertices $i$ and $j$ are adjacent precisely when $(w(i), w(j)) \in D$. A \emph{lettering} of a graph~$G$ consists of a word~$w$, a decoder $D$, and an isomorphism from~$G$ to $\Gamma_D(w)$. The \emph{lettericity} of~$G$, denoted by $\ell(G)$, is the least number of distinct letters in the word of a lettering of~$G$. If a lettering sends the vertex~$v$ of~$G$ to the vertex~$i$ of~$\Gamma_D(w)$, then the letter $w(i)$ \emph{encodes}~$v$; if it sends~$u$ to $i$ and~$v$ to $j$ with $i < j$, then~$u$ is \emph{encoded before}~$v$. The vertices encoded by a single letter $a$ form a clique if $(a,a) \in D$ and an independent set otherwise.

Alecu, Kant\'e, Lozin, and Zamaraev~\cite{alecu:lettericity-of-:} showed that deciding whether $\ell(G) \le k$ is fixed-parameter tractable when parameterized by $k$, and they conjectured that the problem is NP-complete when $k$ is part of the input. Grobler, Morawietz, and Sacher~\cite{grobler:towards-settlin:} studied related problems, which they call retrieval problems, in which a graph is given together with two of three objects (a word, a decoder, and an assignment of letters to the vertices), and the task is to find the third so that the graph has a lettering with this word and decoder in which each vertex is encoded by its assigned letter. They showed that the word and the decoder can be retrieved in polynomial time, while retrieving the assignment of letters is equivalent to the graph isomorphism problem. We prove the conjecture of Alecu, Kant\'e, Lozin, and Zamaraev, in fact for a restricted class of graphs.

\begin{theorem}
\label{thm:np-complete}
Given a graph~$G$ and an integer~$k$, deciding whether $\ell(G) \le k$ is NP-complete, even when~$G$ is a four-colorable comparability graph.
\end{theorem}

While completing this paper, we learned that Fernau, German, and Mann~\cite{fernau:lettericity-is-:} have independently and concurrently proved that this problem is NP-complete for general graphs, by a reduction from monotone not-all-equal 3-SAT. They also show that, under the Exponential Time Hypothesis, the problem cannot be decided in time $2^{o(n)}$ on graphs with $n$ vertices, and they settle the complexity of two extension problems posed by Grobler, Morawietz, and Sacher, in which only the word or only the assignment of letters is given. Our approach is different. We prove an exact formula (Theorem~\ref{thm:formula}) for the lettericity of the graphs our reduction produces, in terms of the greatest number of edge-disjoint paths on four vertices in a bipartite graph. NP-completeness follows from this formula in a few lines, and because these graphs are four-colorable comparability graphs, it holds even within that class.

We write~$P_4$ for the path on four vertices. A \emph{$P_4$-decomposition} of a graph~$G$ is a set of edge-disjoint copies of~$P_4$ in~$G$ that together cover every edge of~$G$; here and throughout, copies of~$P_4$ need not be induced, and distinct copies may share vertices. Teypaz and Rapine~\cite[Theorem~4.3]{teypaz:graph-decomposi:} proved that deciding whether a bipartite graph has a~$P_4$-decomposition is NP-complete, and we reduce from this problem. Their proof is itself a reduction from exact cover by 3-sets (X3C), a classic NP-complete problem~\cite[Section~3.1.2]{garey:computers-and-i:}: given a set~$X$ and a collection~$\mathcal{S}$ of $3$-element subsets of~$X$, the problem is to decide whether some subcollection of~$\mathcal{S}$ partitions~$X$. 

The key quantity in our reduction is the greatest number of pairwise edge-disjoint copies of~$P_4$ in a graph~$G$, which we denote by $\nu(G)$. It is related to lettericity through the incidence graph of~$G$: in a lettering of the incidence graph $I(G)$, a single letter can encode at most three edge-vertices, with equality only if the corresponding edges of~$G$ form a copy of~$P_4$. Consequently, if~$G$ has~$m$ edges, then these vertices are encoded by at least $\lceil (m - \nu(G))/2 \rceil$ letters (Proposition~\ref{pro:lower}). When~$G$ is bipartite, this many letters suffice: given $\nu(G)$ edge-disjoint copies of~$P_4$ in~$G$, we can encode the vertices corresponding to the three edges of each copy by a single letter, and those corresponding to the remaining edges two to a letter (Proposition~\ref{pro:upper}).

The reduction maps a bipartite graph~$G$ to a graph~$J(G)$ obtained from the incidence graph of~$G$ by inflating each vertex corresponding to a vertex of~$G$ by a clique of three vertices and each vertex corresponding to an edge of~$G$ by an independent set of three vertices; two new vertices replacing different vertices of the incidence graph are adjacent precisely when the vertices they replace are adjacent. If~$G$ has~$n$ vertices and~$m$ edges, then~$J(G)$ has $3n + 3m$ vertices and $3n + 18m$ edges.

\begin{theorem}
\label{thm:formula}
For every bipartite graph~$G$ with~$n$ vertices and~$m$ edges,
\[
	\ell(J(G)) = n + \left\lceil \frac{m - \nu(G)}{2} \right\rceil.
\]
\end{theorem}

Applying Theorem~\ref{thm:formula} to the disjoint union $G \uplus G$, which is again bipartite, removes the ceiling. Because~$P_4$ is connected, every copy of~$P_4$ in $G \uplus G$ lies in one of the two copies of~$G$, so $\nu(G \uplus G) = 2\nu(G)$, and Theorem~\ref{thm:formula} gives
\[
	\ell(J(G \uplus G)) = 2n + m - \nu(G)
\]
for every bipartite graph~$G$ with~$n$ vertices and~$m$ edges. Thus, given~$n$ and~$m$,~$\ell(J(G \uplus G))$ and~$\nu(G)$ determine each other.

Section~\ref{sec:incidence} proves the facts about letterings of incidence graphs described above. In Section~\ref{sec:formula}, we prove Theorem~\ref{thm:formula} and then show that Theorem~\ref{thm:np-complete} follows quickly from it.

\section{Incidence graphs}
\label{sec:incidence}

We use the following fact about letterings throughout our arguments.

\begin{observation}
\label{obs:one-switch}
In a lettering of a graph, suppose that the vertices $x_1, x_2, \dots, x_r$ are encoded by the same letter, in this order, and that $z$ is a vertex other than these. Then $z$ is adjacent either to all or to none of the $x_i$ encoded before it, and either to all or to none of the $x_i$ encoded after it.
\end{observation}

We say that a vertex is \emph{encoded between} two others if it is encoded after one of them and before the other. In particular, Observation~\ref{obs:one-switch} shows that a vertex that is adjacent to exactly one of two vertices encoded by the same letter must be encoded between them. This special case is Proposition~2.1 of Ferguson and Vatter~\cite{ferguson:letter-graphs-a:}, and the general statement is implicit in the proof of Proposition~4 of Mandrick and Vatter~\cite{mandrick:bounds-on-the-l:}.

Let~$G$ be a graph with vertex set~$V$ and edge set~$E$. The \emph{incidence graph} of~$G$, denoted by~$I(G)$, is the graph with vertex set
\[
	\hat{V} \cup \hat{E} = \{\hat{v} : v \in V\} \cup \{\hat{e} : e \in E\}
\]
and edge set $\{\hat{v}\hat{e} : v \in e\}$. We refer to the vertices in $\hat{V}$ as \emph{vertex-vertices} and to those in $\hat{E}$ as \emph{edge-vertices}.

We say that a vertex has a \emph{private letter} in a lettering if the letter encoding it encodes no other vertex. We first construct, for bipartite~$G$, a lettering of~$I(G)$ in which every vertex-vertex has a private letter.

\begin{proposition}
\label{pro:upper}
For every bipartite graph~$G$ with~$m$ edges, the incidence graph~$I(G)$ has a lettering in which every vertex-vertex has a private letter and the edge-vertices are encoded by at most ${\lceil (m - \nu(G))/2 \rceil}$ letters.
\end{proposition}

\begin{proof}
Let $X \cup Y$ be a bipartition of~$G$. We first partition the edges of~$G$ into parts, with one part for each letter that will encode edge-vertices. We then label each edge, within its part, as first, middle, or last. Fix $\nu(G)$ pairwise edge-disjoint copies of~$P_4$ in~$G$. The three edges of each of these copies form one part. Because~$G$ is bipartite, each of these copies has one endpoint in~$X$ and the other in~$Y$, and reading it from its endpoint in~$X$ to its endpoint in~$Y$, we label its edges first, middle, and last, in that order. We place the remaining $m - 3\nu(G)$ edges arbitrarily in parts of size two, together with one part of size one if parity requires. In a part of size two, we label one edge first and the other last, and we label the edge in a part of size one first. There are thus $\nu(G)$ parts of size three, $\lfloor (m - 3\nu(G))/2 \rfloor$ parts of size two, and at most one part of size one, for a total of
\[
	\nu(G) + \left\lceil \frac{m - 3\nu(G)}{2} \right\rceil = \left\lceil \frac{m - \nu(G)}{2} \right\rceil.
\]

We encode~$I(G)$ using a word of the form
\[
	w = w_1\, w_X\, w_2\, w_Y\, w_3,
\]
where $w_1$, $w_2$, and $w_3$ encode the edge-vertices $\hat{e}$ for which $e$ is a first, middle, or last edge, respectively, and $w_X$ and $w_Y$ encode the vertex-vertices $\hat{v}$ for which~$v$ lies in~$X$ or in $Y$, respectively. Within each of these five subwords, the vertices may be encoded in any order. Each vertex-vertex has a private letter, and the edge-vertices of each part share a letter, with different letters for different parts, so the edge-vertices are encoded by $\lceil (m - \nu(G))/2 \rceil$ letters, one for each part.

It remains to choose a decoder~$D$. Since the edge-vertices form an independent set, as do the vertex-vertices, the pairs in~$D$ all consist of one letter encoding edge-vertices and one letter encoding a vertex-vertex. We specify these pairs one part at a time, writing~$a$ for the letter that encodes the edge-vertices of the part.

Suppose first that the part has size three, consisting of the first, middle, and last edges $e_1 = v_1 v_2$, $e_2 = v_2 v_3$, and $e_3 = v_3 v_4$ of one of the fixed copies of~$P_4$, where $v_1, v_3 \in X$ and $v_2, v_4 \in Y$. The corresponding seven vertices of~$I(G)$ are encoded in~$w$ in the order
\[
	\hat{e}_1\, \hat{v}_1\, \hat{v}_3\, \hat{e}_2\, \hat{v}_2\, \hat{v}_4\, \hat{e}_3,
\]
up to the order of $\hat{v}_1$ and $\hat{v}_3$ and the order of $\hat{v}_2$ and $\hat{v}_4$. Writing $b_i$ for the letter encoding $\hat{v}_i$, the letters encoding these vertices are, in the same order,
\[
	a\, b_1\, b_3\, a\, b_2\, b_4\, a.
\]
We let $(a,b_1)$, $(a,b_2)$, $(b_3,a)$, and $(b_4,a)$ be the pairs of $D$ involving $a$. These four pairs give precisely the adjacencies of~$I(G)$ between the three edge-vertices $\hat{e}_i$ and the vertex-vertices, as shown in Figure~\ref{fig:P4-encoding}.

\begin{figure}[t]
\begin{footnotesize}
\begin{center}

\begin{tikzpicture}[scale=0.5, baseline=(current bounding box.center)]
	\node (v1) at (0,2) [hollow] {};
	\node (v2) at (2,2) [hollow] {};
	\node (v3) at (0,0) [hollow] {};
	\node (v4) at (2,0) [hollow] {};

	\node (e1) at (1,2) {};
	\node (e2) at (1,1) {};
	\node (e3) at (1,0) {};

	\draw (v1)--(v2)--(v3)--(v4);

	\node[left=1pt of v1] {$v_1$};
	\node[right=1pt of v2] {$v_2$};
	\node[left=1pt of v3] {$v_3$};
	\node[right=1pt of v4] {$v_4$};

	\node[above=1pt of e1] {$e_1$};
	\node[right=1pt of e2] {$e_2$};
	\node[below=1pt of e3] {$e_3$};
\end{tikzpicture}
\qquad\qquad
\begin{tikzpicture}[scale=0.5, baseline=(current bounding box.center)]
	\node (v1) at (0,2) [hollow] {};
	\node (v2) at (2,2) [hollow] {};
	\node (v3) at (0,0) [hollow] {};
	\node (v4) at (2,0) [hollow] {};

	\node (e1) at (1,2) [hollow] {};
	\node (e2) at (1,1) [hollow] {};
	\node (e3) at (1,0) [hollow] {};

	\draw (v1)--(e1)--(v2)--(e2)--(v3)--(e3)--(v4);

	\node[left=1pt of v1] {$\hat v_1$};
	\node[right=1pt of v2] {$\hat v_2$};
	\node[left=1pt of v3] {$\hat v_3$};
	\node[right=1pt of v4] {$\hat v_4$};

	\node[above=1pt of e1] {$\hat e_1$};
	\node[right=1pt of e2] {$\hat e_2$};
	\node[below=1pt of e3] {$\hat e_3$};
\end{tikzpicture}
\qquad\qquad
\begin{tikzpicture}[scale=0.5, baseline=(current bounding box.center)]
	\arc{1}{2}
	\arc{1}{5}
	\arc{3}{4}
	\arc{3}{7}
	\arc{4}{5}
	\arc{6}{7}

	\node (e1) at (1,0) [hollow] {};
	\node (v1) at (2,0) [hollow] {};
	\node (v3) at (3,0) [hollow] {};
	\node (e2) at (4,0) [hollow] {};
	\node (v2) at (5,0) [hollow] {};
	\node (v4) at (6,0) [hollow] {};
	\node (e3) at (7,0) [hollow] {};

	\node[below=2pt of e1] (le1) {$\hat e_1$};
	\node[below=2pt of v1] (lv1) {$\hat v_1$};
	\node[below=2pt of v3] (lv3) {$\hat v_3$};
	\node[below=2pt of e2] (le2) {$\hat e_2$};
	\node[below=2pt of v2] (lv2) {$\hat v_2$};
	\node[below=2pt of v4] (lv4) {$\hat v_4$};
	\node[below=2pt of e3] (le3) {$\hat e_3$};

	\node[below=1pt of le1, letter] {$a$};
	\node[below=1pt of lv1, letter] {$b_1$};
	\node[below=1pt of lv3, letter] {$b_3$};
	\node[below=1pt of le2, letter] {$a$};
	\node[below=1pt of lv2, letter] {$b_2$};
	\node[below=1pt of lv4, letter] {$b_4$};
	\node[below=1pt of le3, letter] {$a$};
	
\end{tikzpicture}

\end{center}
\end{footnotesize}
\caption{On the left, a copy of~$P_4$ in~$G$. In the center, the corresponding vertices of~$I(G)$, which form a copy of $P_7$. On the right, the same vertices as encoded.}
\label{fig:P4-encoding}
\end{figure}
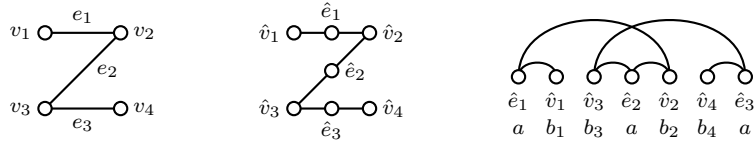

Next suppose that the part has size two, with first edge $e_1 = v_1 v_2$ and last edge $e_2 = v_3 v_4$, where $v_1, v_3 \in X$ and $v_2, v_4 \in Y$; these edges may share an endpoint, in which case $v_1 = v_3$ or $v_2 = v_4$. The corresponding vertices of~$I(G)$ are encoded in~$w$ in the order
\[
	\hat{e}_1\, \hat{v}_1\, \hat{v}_3\, \hat{v}_2\, \hat{v}_4\, \hat{e}_2,
\]
up to the order of $\hat{v}_1$ and $\hat{v}_3$ and the order of $\hat{v}_2$ and $\hat{v}_4$, where a vertex-vertex listed twice is encoded only once. Writing $b_i$ for the letter encoding $\hat{v}_i$, the letters encoding these vertices are, in the same order,
\[
	a\, b_1\, b_3\, b_2\, b_4\, a,
\]
where, if $b_1 = b_3$ or $b_2 = b_4$, this letter appears only once. We again let $(a,b_1)$, $(a,b_2)$, $(b_3,a)$, and $(b_4,a)$ be the pairs of $D$ involving $a$. These four pairs give precisely the adjacencies of~$I(G)$ between the two edge-vertices $\hat{e}_i$ and the vertex-vertices.

Finally, suppose that the part has size one, consisting of the edge $e = v_1 v_2$. Then $\hat{e}$ is encoded before every vertex-vertex, and writing $b_i$ for the letter encoding $\hat{v}_i$, we let $(a,b_1)$ and $(a,b_2)$ be the pairs of $D$ involving $a$. These two pairs give precisely the adjacencies of~$I(G)$ between~$\hat{e}$ and the vertex-vertices.
\end{proof}

The next proposition shows that the edge-vertices of~$I(G)$ cannot be encoded by fewer than the $\lceil (m - \nu(G))/2 \rceil$ letters used in Proposition~\ref{pro:upper}, even when~$G$ is not bipartite and the vertex-vertices do not have private letters.

\begin{proposition}
\label{pro:lower}
Let~$G$ be a graph with~$m$ edges, and consider a lettering of~$I(G)$. Each letter encodes at most three edge-vertices, and if it encodes three, then the corresponding edges of~$G$ form a copy of~$P_4$. Consequently, the edge-vertices are encoded by at least $\lceil (m - \nu(G))/2 \rceil$ letters.
\end{proposition}

\begin{proof}
Suppose that the edge-vertices $\hat{e}_1, \hat{e}_2, \dots, \hat{e}_r$ are encoded by the same letter, in this order, where $r \ge 2$. For each $2 \le i \le r$, the edges $e_{i-1}$ and $e_i$ are distinct, so $e_i$ has an endpoint $v_i$ that is not an endpoint of $e_{i-1}$. The vertex-vertex $\hat{v}_i$ is then adjacent to $\hat{e}_i$ but not to $\hat{e}_{i-1}$, so by Observation~\ref{obs:one-switch}, it is encoded between them, and it is adjacent to every $\hat{e}_j$ encoded after it and to none of those encoded before it. It follows that the vertex-vertices $\hat{v}_2, \dots, \hat{v}_r$ are distinct (they are encoded in different places) and are all adjacent to $\hat{e}_r$. But $e_r$ has only two endpoints, so we must have $r \le 3$.

Now suppose that $r = 3$, so that $e_3 = v_2 v_3$. For $i = 1, 2$, let $u_i$ be an endpoint of $e_i$ that is not an endpoint of $e_{i+1}$. The same argument, read backwards, shows that $\hat{u}_1$ and $\hat{u}_2$ are distinct and are both adjacent to $\hat{e}_1$, so $e_1 = u_1 u_2$. Neither $v_2$ nor $v_3$ is an endpoint of $e_1$, because $\hat{v}_2$ and $\hat{v}_3$ are adjacent to none of the $\hat{e}_j$ encoded before them, so $e_1$ and $e_3$ share no endpoint. Consequently, $u_2$ and $v_2$ are distinct, and as both are endpoints of $e_2$, we have $e_2 = u_2 v_2$. Thus, $e_1$, $e_2$, and $e_3$ form the path $u_1 u_2 v_2 v_3$.

Finally, let $g$ be the number of letters that encode edge-vertices, and let $t$ be the number of these that encode three. The $t$ corresponding copies of~$P_4$ are edge-disjoint, since each edge-vertex is encoded by only one letter, so $t \le \nu(G)$. Counting the edge-vertices according to their letters gives
\[
	m \le 3t + 2(g - t) = 2g + t \le 2g + \nu(G),
\]
and since $g$ is an integer, it follows that $g \ge \lceil (m - \nu(G))/2 \rceil$.
\end{proof}

\section{An exact formula}
\label{sec:formula}

For a bipartite graph~$G$ with~$n$ vertices and~$m$ edges, Propositions~\ref{pro:upper} and~\ref{pro:lower} show that the least number of letters in a lettering of~$I(G)$ in which every vertex-vertex has a private letter is $n + \lceil (m - \nu(G))/2 \rceil$. However, a lettering of~$I(G)$ in which the vertex-vertices do not have private letters can use fewer letters. For example,~$I(K_2)$ is the path on three vertices, and its two vertex-vertices can share a letter, so $\ell(I(K_2)) = 2$, while $n + \lceil (m - \nu(K_2))/2 \rceil = 2 + 1 = 3$. To prevent such savings, we use twin sets.

A set~$T$ of vertices of a graph is a \emph{twin set} if~$T$ is a clique or an independent set and every vertex outside~$T$ is adjacent either to all of the vertices of~$T$ or to none of them. (In the language of modular decomposition, a twin set is a module that is a clique or an independent set.) By the following proposition, we may assume that a twin set of at least three vertices is encoded by a single letter. Alecu, Kant\'e, Lozin, and Zamaraev~\cite[Lemma~17]{alecu:lettericity-of-:} showed that deleting a vertex from a twin set of at least four vertices does not change the lettericity, and our argument is essentially theirs.

\begin{proposition}
\label{pro:twins}
Let~$T$ be a twin set of at least three vertices in a graph. Every lettering of the graph can be modified, without introducing new letters, so that the vertices of~$T$ are encoded by a single letter, while the vertices outside~$T$ keep their letters.
\end{proposition}

\begin{proof}
Fix a lettering of the graph, with word~$w$ and decoder~$D$. Choose three vertices $v_1$, $v_2$, and~$v_3$ of~$T$, encoded in this order, and let~$a$ be the letter encoding~$v_2$. We give the proof when~$T$ is a clique; a symmetric argument gives the case where~$T$ is an independent set.

We first show that we may assume that $(a,a) \in D$. If~$a$ encodes only~$v_2$, we may simply add $(a,a)$ to~$D$. Suppose instead that~$a$ also encodes a vertex $u \ne v_2$, and suppose, for the sake of contradiction, that $(a,a) \notin D$. Then~$u$ and~$v_2$ are nonadjacent, so $u \notin T$, and thus~$u$ is nonadjacent to every vertex of~$T$. Hence~$v_1$ and~$v_3$ are each adjacent to exactly one of~$v_2$ and~$u$, so by Observation~\ref{obs:one-switch}, both are encoded between~$v_2$ and~$u$. However, this is impossible, because~$v_1$ is encoded before~$v_2$ and~$v_3$ after it.

With $(a,a) \in D$, we modify the lettering by first deleting the occurrences encoding the vertices of $T \setminus \{v_2\}$ and then replacing the occurrence of~$a$ encoding~$v_2$ by $|T|$ consecutive copies of~$a$, which encode the vertices of~$T$. The vertices outside~$T$ keep their letters and are encoded in the same order as before, so the adjacencies among them are unchanged. The vertices of~$T$ still form a clique because $(a,a) \in D$, and each vertex outside~$T$ is encoded on the same side of all of them as it was of~$v_2$, so, as~$T$ is a twin set, its adjacencies to~$T$ are also unchanged. The result is therefore again a lettering of the graph, and it uses no letters that the original did not.
\end{proof}

We now inflate each vertex-vertex of~$I(G)$ by a clique of three vertices and each edge-vertex by an independent set of three vertices, obtaining the graph~$J(G)$. These $n + m$ sets of three vertices form pairwise disjoint twin sets, so by applying Proposition~\ref{pro:twins} to each of them in turn, we may assume that each is encoded by a single letter, and we consider only such letterings of~$J(G)$ in what follows.

Note that the twin sets replacing vertex-vertices are cliques, but vertex-vertices are not adjacent to each other in~$I(G)$, so no two vertex-vertices can share a letter. Moreover, the twin sets replacing edge-vertices are independent sets, so no vertex-vertex can share a letter with an edge-vertex. Thus, choosing one vertex from each twin set turns such a lettering of~$J(G)$ into a lettering of~$I(G)$, with no additional letters, in which every vertex-vertex has a private letter.

\begin{proof}[Proof of Theorem~\ref{thm:formula}]
By the discussion above, a lettering of~$J(G)$ yields a lettering of~$I(G)$ with no additional letters in which every vertex-vertex has a private letter, so $\ell(J(G)) \ge n + \lceil (m - \nu(G))/2 \rceil$. Conversely, Proposition~\ref{pro:upper} provides a lettering of~$I(G)$ with at most $n + \lceil (m - \nu(G))/2 \rceil$ letters in which every vertex-vertex has a private letter. We may assume that $(a,a)$ lies in its decoder precisely when~$a$ encodes a vertex-vertex, because the edge-vertices are pairwise nonadjacent and each vertex-vertex has a private letter. Replacing each occurrence of a letter in the word by three consecutive copies of it, and letting these copies encode the three vertices of~$J(G)$ replacing the vertex encoded by that occurrence, gives a lettering of~$J(G)$ with no additional letters, so $\ell(J(G)) \le n + \lceil (m - \nu(G))/2 \rceil$.
\end{proof}

We now prove Theorem~\ref{thm:np-complete}, using Theorem~\ref{thm:formula} for NP-hardness. For the restriction to four-colorable comparability graphs, recall that a comparability graph is a graph whose edges are the comparable pairs of some partial order on its vertices.

\begin{proof}[Proof of Theorem~\ref{thm:np-complete}]
The problem lies in NP because a lettering of the input graph with at most~$k$ letters has polynomial size once unused letters are discarded, and it can be checked by examining every pair of vertices. For NP-hardness, we reduce from the problem of deciding whether a bipartite graph has a~$P_4$-decomposition. Let~$G$ be a bipartite graph with~$n$ vertices and~$m$ edges, and write $q = \lfloor m/3 \rfloor$. We show that~$G$ has a~$P_4$-decomposition precisely when $\ell(J(G)) \le n + q$; since~$J(G)$ has~${3n + 3m}$ vertices and can be constructed in polynomial time, this gives a polynomial-time reduction. Because each copy of~$P_4$ has three edges, $\nu(G) \le q$, and because a~$P_4$-decomposition is a set of edge-disjoint copies of~$P_4$ covering all $m$ edges, $G$ has one precisely when $3\nu(G) = m$, that is, when $m = 3q$ and $\nu(G) = q$. On the other hand, Theorem~\ref{thm:formula} shows that $\ell(J(G)) \le n + q$ precisely when $\lceil (m - \nu(G))/2 \rceil \le q$, and since $q$ is an integer, this holds precisely when $m - \nu(G) \le 2q$. As~${m \ge 3q}$ and $\nu(G) \le q$, this inequality also holds precisely when $m = 3q$ and $\nu(G) = q$.

Finally, we show that each graph~$J(G)$ is a four-colorable comparability graph, so that the reduction stays within this class. Order the three vertices replacing each vertex-vertex~$\hat{v}$ as a chain, and place all of them below every vertex replacing an edge-vertex~$\hat{e}$ with $v \in e$. This relation is transitive: a vertex lies below a vertex of a chain only if it belongs to the same chain, all vertices of a chain lie below the same vertices replacing edge-vertices, and no vertex lies above a vertex replacing an edge-vertex. Its comparable pairs are precisely the edges of~$J(G)$. Giving the three vertices of each chain the colors $1$, $2$, and~$3$ and every vertex replacing an edge-vertex the color~$4$ yields a proper coloring of~$J(G)$.
\end{proof}

\subsection*{Tool use disclosure}

This work was carried out using ChatGPT Astra 6 Pro and Claude Opus 5.5. An initial proof of NP-completeness was produced by ChatGPT in response to the author's request. The author subsequently worked with both models to substantially simplify the argument, and to extract the $P_4$-counting formula. The manuscript was written primarily by the author, with assistance from the models in checking and copyediting. The author takes full responsibility for the results.

\setlength{\bibsep}{4pt}

%
%
%


\def\cprime{$'$}

\end{document}